\documentclass{amsart}
\usepackage[utf8]{inputenc}

\usepackage{hyperref}
\usepackage{epsfig}
\usepackage{appendix}
\usepackage{amsfonts,amssymb,amsmath,amsthm,amscd}
\usepackage{latexsym}
\usepackage{mathtools}
\usepackage{graphicx}
\usepackage[all, cmtip]{xy}

\newtheorem{innercustomgeneric}{\customgenericname}
\providecommand{\customgenericname}{}
\newcommand{\newcustomtheorem}[2]{%
  \newenvironment{#1}[1]
  {%
   \renewcommand\customgenericname{#2}%
   \renewcommand\theinnercustomgeneric{##1}%
   \innercustomgeneric
  }
  {\endinnercustomgeneric}
}

\newcustomtheorem{thm}{Theorem}
\newcustomtheorem{cor}{Corollary}

\theoremstyle{plain}
		\newtheorem{theorem}{Theorem}[section]
		
		\newtheorem{def-thm}{Definition-Theorem}[section]
		\newtheorem{lemma}[theorem]{Lemma}
		\newtheorem{corollary}[theorem]{Corollary}
		\newtheorem{proposition}[theorem]{Proposition}

		\newtheorem*{theorem*}{Theorem}

\theoremstyle{definition}
		\newtheorem{definition}[theorem]{Definition}
		
		\newtheorem{question}[theorem]{Question}
		
\theoremstyle{remark}

\newcommand{\A}{{\mathbb{A}}}

\newcommand{\C}{{\mathbb{C}}}

\renewcommand{\P}{{\mathbb{P}}}
\newcommand{\Q}{{\mathbb{Q}}}

\newcommand{\V}{{\mathbb{V}}}

\newcommand{\Z}{{\mathbb{Z}}}

\newcommand{\Lcal}{{\mathcal{L}}}

\newcommand{\Ocal}{{\mathcal{O}}}

\newcommand{\mfrak}{{\mathfrak{m}}}

\newcommand{\pfrak}{{\mathfrak{p}}}
\newcommand{\qfrak}{{\mathfrak{q}}}

\newcommand{\id}{{\textup{id}}}

\DeclareMathOperator{\Aut}{Aut}

\DeclareMathOperator{\End}{End}

\DeclareMathOperator{\Fin}{Fin}

\DeclareMathOperator{\GL}{GL}

\DeclareMathOperator{\Hom}{Hom}

\DeclareMathOperator{\sheafhom}{\mathcal{H}\kern -.5pt \emph{om}}
\DeclareMathOperator{\Img}{Im}

\DeclareMathOperator{\Mat}{Mat}

\DeclareMathOperator{\Pic}{Pic}

\DeclareMathOperator{\PGL}{PGL}

\DeclareMathOperator{\SL}{SL}

\DeclareMathOperator{\Char}{char}

\DeclareMathOperator{\Span}{Span}
\DeclareMathOperator{\Spec}{Spec}

\DeclareMathOperator{\Tor}{Tor}

\DeclareMathOperator{\Per}{Per}

\DeclareMathOperator{\Preper}{Preper}

\usepackage{booktabs}

\title{On periodic orbits of polynomial maps}
\author{Junho Peter Whang}
\address{Department of Mathematical Sciences and RIM,
Seoul National University}
\email{jwhang@snu.ac.kr}
\date{\today}

\begin{document}

\begin{abstract}
We prove the existence of an effective universal upper bound for the order of any integral periodic orbit of any integral algebraic dynamical system in a fixed ambient space. Using this, we demonstrate the decidability of periodicity in arbitrary finitely generated algebraic dynamical systems over fields of characteristic zero.
\end{abstract}

\maketitle


\section{Introduction}\label{sect:1}
\subsection{\unskip}
This paper concerns effective universal bounds for orders of integral periodic orbits in algebraic dynamics. Suppose that $S$ is a collection of endomorphisms of a set $X$, generating a monoid $\langle S\rangle$ under composition. We shall say that a point $x\in X$ is \emph{$S$-periodic} if its $S$-orbit $\Ocal_S(x)=\{f(x):f\in\langle S\rangle\}$ is a finite set and $\langle S\rangle$ acts on $\Ocal_S(x)$ by permutations. For example, if $S$ consists of a single map $f$, then a point $x\in X$ is $S$-periodic if and only if $f^k(x)=x$ for some $k\geq1$. To motivate our main result, we begin with a special case that may be of general interest.

\begin{theorem}
\label{affine}
Fix $n\geq1$. There is an effective universal constant $C(n)$ such that, for an arbitrary set $S$ of polynomial maps from $\Z^n$ to itself with integer coefficients, and any $S$-periodic point $x\in\Z^n$, we have
$$|\Ocal_S(x)|\leq C(n).$$
In particular, there is a polynomial-time algorithm to decide, for $x\in\Z^n$ and a finite set $S$ of polynomial endomorphisms of $\Z^n$ over $\Z$, whether or not $x$ is $S$-periodic.
\end{theorem}

An explicit computation of an upper bound $C(n)$ is given in Section \ref{sect:2.4}. Clearly, the same universal bound $C(n)$ applies to the integral periodic orbits 
for all sets of endomorphisms of any scheme embeddable as a closed subscheme of the affine $n$-space $\A^n$ over $\Z$. Let $B(n)$ denote the minimum value of possible $C(n)$ for which Theorem \ref{affine} holds. Determining the exact value of $B(n)$ for all $n\geq1$ remains an open problem. We remark that $B(m+n)\geq B(m)B(n)$ for all $m,n$, so $B(n)$ has rapid growth as $n\to\infty$. Theorem \ref{affine} is a special case of the following.

\begin{theorem}
\label{mainthm}
    Let $V$ be a separated scheme flat of finite presentation over a finitely presented domain $k$ of characteristic zero. There is an effective universal constant $C(V,k)$ such that, for any subscheme $X$ of $V$, any set $S$ of endomorphisms of $X/k$, and any $S$-periodic point $x\in X(k)$, we have
    $$|\Ocal_S(x)|\leq C(V,k).$$
\end{theorem}

Here, by a subscheme of $V$ we mean a closed subscheme of an open subscheme of $V$. Since an endomorphism of a subscheme of $V$ need not always extend to an endomorphism of $V$, Theorem \ref{mainthm} is stronger than the version that only considers endomorphisms of $V/k$. If $V$ is proper over a number ring $k$, then writing $F$ for the fraction field of $k$ we have $X(k)=X(F)$ for any closed subscheme $X$ of $V$, so Theorem \ref{mainthm} implies universal boundedness of orders of rational periodic orbits of algebraic dynamical systems supported on closed subschemes of $V$ (cf.~paragraph following Question \ref{mainquestion}). Another notable special case of Theorem \ref{mainthm} is where $V$ is the total space of a parametric family of schemes given by a morphism $\pi:V\to B$. Theorem \ref{mainthm} implies in this case that the maximum orders of integral periodic orbits on the fibers $\pi^{-1}(p)$ for $p\in B(k)$ are uniformly bounded.

Theorem \ref{mainthm} appears to be new if the relative dimension of $V/k$ is at least $2$. The case $V=\P^1$ for single maps (i.e.~$|S|=1$) is due to Morton-Silverman \cite{ms}. Of note is the method of reduction modulo primes used therein; see also \cite[Section 2.6]{silverman} and \cite{li, ms, ms2, narkiewicz, pezda, zieve, hutz}. The same method has been used to compute rational torsion subgroups of elliptic curves (See e.g.~\cite[Section VIII7]{silverman2}). Historically, it was used by Minkowski \cite{minkowski} to bound the orders of finite subgroups of $\GL_n(\Z)$, and also used to prove Selberg's lemma (See \cite{serre}). Also using this method, Fakhruddin \cite{fakhruddin} observed the boundedness (without universality) of periods for individual maps on integral points of varieties. Our work relies on the same method of reduction modulo primes, with a slight twist. Instead of analyzing each individual dynamical system, we consider all closed subschemes associated to finite sets of integral points in the ambient space $V$. We then obtain a uniform bound on their automorphism groups, in terms of counts of points on $V$ with values in certain finite rings determined by $(V,k)$. The key ingredient here is an effective form of nonlinear Selberg's lemma due to Bass--Lubotzky \cite{bl}; reductions modulo primes appear in the proof this latter result. Theorem \ref{mainthm} is a consequence of the following.

\begin{theorem}
\label{mainthm'}
    Let $V$ be a separated scheme flat of finite presentation over a finitely presented domain $k$ of characteristic zero. There is an effective universal constant $C(V,k)$ such that, for every finite set $\Ocal\subseteq V(k)$, we have
    $$|\Aut(Z_\Ocal)|\leq C(V,k)$$
    where $Z_\Ocal$ denotes the reduced closed subscheme of $V$ associated to $\Ocal$.
\end{theorem}

Implicit in the effectivity of $C(V,k)$ in Theorems \ref{mainthm} and \ref{mainthm'} is the assumption that $(V,k)$ is explicitly given, in the following sense. First, we assume that $k$ is given as a quotient of a polynomial ring with coefficients in $\Z$, such that a finite set of generators for the kernel of the quotient map is specified. This means that finitely generated subrings of $\C$ such as $\Z[\pi,e]$ would be admissible \emph{modulo} difficult problems in transcendental number theory. As for $V$, we assume that $V$ is given as an explicit finite union of open affine schemes $U_i$ with affine overlaps $U_i\cap U_j$, each of which is the spectrum of an explicitly finitely presented $k$-algebra, such that the gluing morphisms $U_i\cap U_j\to U_i$ are induced by explicitly given morphisms of $k$-algebras. In practice, one may assume that $V$ is defined by an explicit finite collection of polynomials with $k$-coefficients with $\A^n$ or $\P^n$ as ambient space. Theorem \ref{mainthm} implies the following decidability result for periodic points on algebraic varieties (i.e.~reduced separated schemes of finite type over fields).

\begin{theorem}
\label{mainthm3}
Let $S$ be a finite set of endomorphisms of an algebraic variety $V/\bar\Q$. There is an algorithm to decide, given $x\in V(\bar\Q)$, whether or not $x$ is $S$-periodic.
\end{theorem}

As discussed in the preceding paragraph, we may replace $\bar\Q$ by $\C$ or arbitrary fields of characteristic zero modulo considerations in transcendental number theory. In several cases (e.g. if $S$ consists of a polarized endomorphism of a projective variety), stronger results than Theorem \ref{mainthm3} have been available through the theory of heights, starting with Northcott \cite{northcott}; see \cite[Section 4]{survey} for a survey. For example, the method of height functions often entails that the (pre)periodic points of bounded degree over $\Q$ are finite and effectively bounded in height. Note however that such finiteness of rational periodic points does not always hold, even for endomorphisms of the affine space $\A^n$ with $n\geq2$.

In Theorem \ref{mainthm3}, of particular interest is the case where $S$ is a finite symmetric set of generators for a group $G$ of automorphisms of $V$. The periodic points of $V$ in this case are precisely those with finite orbit. There are numerous discrete group actions on varieties whose finite orbits carry special significance. For instance, on moduli spaces of local systems over a fixed Riemann surface $\Sigma$, the points with finite orbit under the mapping class group action correspond to the Fuchsian systems of differential equations on $\Sigma$ with algebraic isomonodromic deformations; see \cite{ch} for details (cf.~\cite{bgmw, ll}). Theorem \ref{mainthm3} shows that, in all examples, the property of having a finite group orbit is a computable invariant of a given point in the variety. In particular, any condition characterizing the periodic points on the variety must be decidable. Theorem \ref{mainthm3} leaves open the following question.

\begin{question}
\label{mainquestion}
Given a finite set $S$ of endomorphisms of an algebraic variety $V/\bar\Q$, is the subset of points with finite $S$-orbit decidable?
\end{question}

More generally, one may ask if the Zariski closure of the $S$-orbit of a point is always computable. Finally, we remark that the results of this paper are in a sense orthogonal to the well-known problem of uniformly bounding the periods of rational points for endomorphisms over $\Q$ of given degree on $\P^n$. More precisely, our work considers the family of sets of arbitrary endomorphisms that have ``good reduction'' modulo a fixed set of primes, whereas the conjectures in \cite{ms, poonen} concern the family of endomorphisms of fixed degree but varying sets of good reduction.

\subsection{Acknowledgements}
I thank Abhishek Oswal, Joseph H.~Silverman, Peter Sarnak, and Daniel Litt for valuable discussions and comments. This work was supported by the Samsung Science and Technology Foundation under Project Number SSTF-BA2201-03.

\section{Proofs of the main results}\label{sect:2}

This section is organized as follows. In Section \ref{sect:2.1}, we record and sketch the proof of a result of Bass-Lubotzky \cite{bl} on automorphism groups of schemes, which we view as a nonlinear analogue of Selberg's lemma. In Section \ref{sect:2.2}, we use the result of Bass-Lubotzky to give a proof of Theorems \ref{mainthm'} and \ref{mainthm}. Then, in Section \ref{sect:2.3}, we establish Theorem \ref{mainthm3} as a corollary of Theorem \ref{mainthm}. We derive Corollary \ref{affine} together with a compuation of an effective bound $C(n)$ in \ref{sect:2.4}.

\subsection{Selberg-Bass-Lubotzky lemma} \label{sect:2.1}
Let $k$ be a finitely presented domain of characteristic zero. Let $Z$ be a scheme flat of finite presentation over $k$. Given a finite set $M$ of closed points of $Z$, let us write
    $$A_{M}=\prod_{x\in M}\Ocal_{Z,x}\quad\text{and}\quad J_{M}=\prod_{x\in X}\mfrak_{Z,x}$$
where $\Ocal_{Z,x}$ denotes the local ring of $Z$ at $x$ and $\mfrak_{Z,x}$ is the maximal ideal of $\Ocal_{Z,x}$, with residue field $\kappa(x)=\Ocal_{Z,x}/\mfrak_{Z,x}$. We shall say that $M$ has \emph{residue characteristic $p$} if $\Char\kappa(x)=p$ for every $x\in M$.

\begin{definition} \cite[p.4]{bl}
Let $M$ be a finite set of closed points in $Z$. We shall say that $M$ is \emph{effective} if there is a finite affine open covering $(U_i)_{i\in I}$ of $Z$ such that the natural morphism
    $\Ocal_Z(U_i)\to\prod_{x\in M\cap U_i}\Ocal_{Z,x}$
    is injective for every $i\in I$.
\end{definition}

\begin{proposition}
\label{prop1}
    Suppose that $M$ is a finite effective set of closed points in $Z$ with residue characteristic $p>0$. Then the order of any torsion element in $$\Gamma_M=\ker(\Aut(Z/k)\to\Aut Z(A_{M}/J_{M}^2))$$ is a power of $p$. In particular, if $M'$ is a finite effective set of closed points in $Z$ with residue characteristic $p'\neq p$, then $\Gamma_M\cap\Gamma_{M'}$ is a torsionfree normal subgroup of finite index in $\Aut(Z/k)$.
\end{proposition}

\begin{proof}
The proof is given in \cite[pp.4-5]{bl}, at least in the case $k=\Z$. For the sake of completeness, we reproduce the proof here, with minor modifications. First, since the group $\Gamma_M$ fixes $Z(A_M/J_M^2)$, it fixes each point $x\in M$ and hence acts on $\Ocal_{Z,x}$ with trivial action on $\Ocal_{Z,x}/\mfrak_{Z,x}^2$. Suppose that $f\in\Gamma_M$ is a torsion element. By effectivity of $M$, to show that $f$ has $p$-power order it suffices to show that it acts with $p$-power order on $\Ocal_{Z,x}$ for each $x\in M$. Since $\bigcap_{r\geq1}{\mfrak_{Z,x}^r}=0$, it suffices that $f$ acts with $p$-power order on the finite ring $\Ocal_{Z,x}/\mfrak_{Z,x}^r$ for all $r$, with the case $r\leq 2$ being trivially true by our assumption on $\Gamma_M$. Let us now view $f$ as a $\Z$-linear endomorphism of $\Ocal_{Z,x}/\mfrak_{Z,x}^r$ with $r\geq3$. Note that $g=f-1$ is a nilpotent endomorphism of $\Ocal_{Z,x}/\mfrak_{Z,x}^r$. Indeed, note that $(f-1)(\mfrak_{Z,x}^{s})\subseteq\mfrak_{Z,x}^{s+1}$ for every $s\geq2$. Choosing $N\gg0$ so that $g^{p^N}=0$, we have
$$
f^{p^{N+r-1}}=((1+g)^{p^N})^{p^{r-1}}=\left(1+ph\right)^{p^{r-1}}=1
$$
for some $h\in\End_{\Z}(\Ocal_{Z,x}/\mfrak_{Z,x}^r)$, where the last equality follows from the fact that $\Ocal_{Z,x}/\mfrak_{Z,x}^r$ is a $\Z/p^r\Z$-module. This proves the desired result.
\end{proof}

\subsection{Configurations of integral points}\label{sect:2.2}
We now prove Theorems \ref{mainthm'} and \ref{mainthm}.

\begin{thm}{\ref{mainthm'}}
    Let $V$ be a separated scheme flat of finite presentation over a finitely presented domain $k$ of characteristic zero. There is an effective universal constant $C(V,k)$ such that, for every finite set $\Ocal\subseteq V(k)$, we have
    $$|\Aut(Z_\Ocal)|\leq C(V,k)$$
    where $Z_\Ocal$ denotes the reduced closed subscheme of $V$ associated to $\Ocal$.
\end{thm}

\begin{proof}
Let $\Ocal\subseteq V(k)$ be a finite set. We view $y\in \Ocal$ as a section $y:\Spec k\to V$, and let $Z_\Ocal=\bigcup_{y\in O}y(\Spec k)$ be the union of the images of $y\in O$, endowed with the structure of a reduced closed subscheme of $V$. Here, separatedness of $V/k$ ensures that $y(\Spec k)\subseteq V$ is closed for $y\in \Ocal$. Note that $\Aut(Z_\Ocal/k)$ is finite.

Given any maximal prime $\qfrak$ of $k$, let us identify $Z_\Ocal(\kappa(\qfrak))$ with the set of closed points $x\in Z_\Ocal$ with residue field $\kappa(x)$ isomorphic to $\kappa(\qfrak)=k/\qfrak$. Note that $Z_\Ocal(\kappa(\qfrak))$ is an effective subset of $Z_\Ocal$ for any choice of $\qfrak$, since $Z_\Ocal(\kappa(\qfrak))$ meets all irreducible components of $Z_\Ocal$. We claim that the cardinality of the finite set $Z_\Ocal(A_{Z_\Ocal(\kappa(\qfrak))}/J_{Z_\Ocal(\kappa(\qfrak))}^2)$ can be bounded independently of $\Ocal$. First, note that
$$Z_\Ocal(A_{Z_\Ocal(\kappa(\qfrak))}/J_{Z_\Ocal(\kappa(\qfrak))}^2)\subseteq V(A_{Z_\Ocal(\kappa(\qfrak))}/J_{Z_\Ocal(\kappa(\qfrak))}^2).$$
Observe next that the coefficient ring $A_{Z_\Ocal(\kappa(\qfrak))}/J_{Z_\Ocal(\kappa(\qfrak))}^2$ is a quotient of the finite ring $A_{V(\kappa(\qfrak))}/J_{V(\kappa(\qfrak))}^2$, where $V(\kappa(\qfrak))$ is the finite set of closed points in $V$ with residue field isomorphic to $\kappa(\qfrak)$, and where we denote
$A_{V(\kappa(\qfrak))}=\prod_{x\in V(\kappa(\qfrak))}\Ocal_{V,x}$ and $J_{V(\kappa(\qfrak))}=\prod_{x\in V(\kappa(\qfrak))}\mfrak_{V,x}$. Thus, we conclude that
$$\textstyle|Z_\Ocal(A_{Z_\Ocal(\kappa(\qfrak))}/J_{Z_\Ocal(\kappa(\qfrak))}^2)|\leq \max_{R\in T(\qfrak)}|V(R)|$$
where $T(\qfrak)$ denotes the set of all quotients of $A_{V(\kappa(\qfrak))}/J_{V(\kappa(\qfrak))}^2$. This establishes the claim.

Let us fix now two maximal primes $\pfrak,\pfrak'$ of $k$ with coprime residue characteristic, independently of $\Ocal$. By Proposition \ref{prop1}, the kernel of
$$\Aut(Z_\Ocal/k)\to\Aut Z_\Ocal(A_{Z_\Ocal(\kappa(\pfrak))}/J_{Z_\Ocal(\kappa(\pfrak))}^2)\times\Aut Z_\Ocal(A_{Z_\Ocal(\kappa(\pfrak'))}/J_{Z_\Ocal(\kappa(\pfrak'))}^2)$$
must be torsion-free, and hence trivial since $\Aut(Z_\Ocal/k)$ is finite. By the previous paragraph, we obtain the bound
$|\Aut(Z_\Ocal/k)|\leq C(V,k)$
with
$$\textstyle C(V,k)=(\max_{R\in T(\pfrak)}|V(R)|!)(\max_{R\in T(\pfrak')}|V(R)|!)$$
which depends only on $V/k$ and the choices of $\pfrak,\pfrak'$, as desired.
\end{proof}

\begin{thm}{\ref{mainthm}}
    Let $V$ be a separated scheme flat of finite presentation over a finitely presented domain $k$ of characteristic zero. There is an effective universal constant $C(V,k)$ such that, for any subscheme $X$ of $V$, any set $S$ of endomorphisms of $X/k$, and any $S$-periodic point $x\in X(k)$, we have
    $$|\Ocal_S(x)|\leq C(V,k).$$
\end{thm}

\begin{proof}
Suppose that $X$ is a subscheme of $V$, $S$ is a set of endomorphisms of $X/k$, and $x\in X(k)$ is an $S$-periodic point. Let $\Ocal=\Ocal_S(x)$. Since an immersion is a monomorphism in the category of schemes, the natural map $X(k)\to V(k)$ is injective, and we shall also write $\Ocal$ for the image of $\Ocal_S(x)$ in $V$. Let $Z_{\Ocal}$ (resp. $Z_\Ocal'$) denote the reduced closed subscheme of $V$ (resp.~$X$) defined by $\bigcup_{y\in\Ocal}y(\Spec k)$. Then the immersion $X\to V$ induces an isomorphism $Z_\Ocal'\simeq Z_\Ocal$. Then by Theorem \ref{mainthm'} we have
$$|\Ocal_S(x)|\leq|\Aut(Z'_{\Ocal}/k)|=|\Aut(Z_{\Ocal}/k)|\leq C(V,k)$$
where $C(V,k)$ is a constant determined solely by $(V,k)$, noting that each member of $\langle S\rangle$ induces an automorphism of the scheme $Z_{\Ocal}'/k$. Alternatively, we can use the observation that, with the choice of formula for $C(V,k)$ given in the proof of Theorem \ref{mainthm'}, we have $C(X,k)\leq C(V,k)$.
\end{proof}

\subsection{Decidability of periodicity}\label{sect:2.3}
We now prove Theorem \ref{mainthm3}.

\begin{thm}{\ref{mainthm3}}
Let $S$ be a finite set of endomorphisms of an algebraic variety $V/\bar\Q$. There is an algorithm to decide, given $x\in V(\bar\Q)$, whether or not $x$ is $S$-periodic.
\end{thm}

\begin{proof}
Let $V$ be an algebraic variety over $\bar\Q$. Let $S$ be a finite set of endomorphisms of $V/\bar\Q$. Let $x\in V(\bar\Q)$. We can effectively determine, by spreading, a finitely generated subring $k$ of $\bar\Q$ such that $(V,S)$ descends to a flat model over $k$ and $x$ descends to a point of $V(k)$. By Theorem \ref{mainthm}, there is an effective integer constant $C(V,k)\geq1$ such that, if $x$ is $S$-periodic, then $|\Ocal_S(x)|\leq C(V,k)$. Consider the ascending chain of finite sets $O_1\subseteq O_2\subseteq\cdots$ where $O_1=\{x\}$ and $$O_{n+1}=O_n\cup\{f(y):y\in O_n,f\in S\}$$ for $n\geq1$. Note that $\Ocal_S(x)=\bigcup_{n\geq1}O_n$, and either:
\begin{enumerate}
    \item $O_{n+1}=O_n$ for some $n\leq C(V,k)$ in which case $O_n=\Ocal_S(x)$ is finite and $S$-periodicity of $x$ can be decided by testing each $f\in S$ on $\Ocal_S(x)$, or
    \item $|\Ocal_S(x)|\geq |O_{C(V,k)}|>C(V,k)$ and hence $x$ is not $S$-periodic.
\end{enumerate} This demonstrates that $S$-periodicity of $x$ is decidable.
\end{proof}

\subsection{Orbits on affine spaces} \label{sect:2.4}
We now prove a quantitative form of Theorem \ref{affine}.

\begin{theorem}
Fix $n\geq1$. There is an effective universal constant $C(n)$ such that, for an arbitrary set $S$ of polynomial maps from $\Z^n$ to itself with integer coefficients, and any $S$-periodic point $x\in\Z^n$, we have
$|\Ocal_S(x)|\leq C(n)$. In fact, we may take $$C(n)=(2^{n(n+2)2^n}!)(3^{n(n+2)3^n}!).$$
In particular, there is a polynomial-time algorithm to decide, for $x\in\Z^n$ and a finite set $S$ of polynomial endomorphisms of $\Z^n$ over $\Z$, whether or not $x$ is $S$-periodic.
\end{theorem}

\begin{proof}
The first part is a special case of Theorem \ref{mainthm}. Let $B(n)$ denote the optimal constant. By the proof of Theorem \ref{mainthm}, we have $$B(n)\leq\textstyle (\max_{R\in T(2)}|\A^n(R)|!)(\max_{R\in T(3)}|\A^n(R)|!)$$
where $T(p)$ for each prime $p$ denotes the set of all quotients of the finite ring
$$A_p=\prod_{(a_1,\dots,a_n)\in \{0,\dots,p-1\}^n}\frac{\Z[x_1,\dots,x_n]}{(p,x_1-a_1,\dots,x_n-a_n)^2}.$$
Each term in the above product has cardinality $p^{n+2}$, and there are $p^n$ terms so $|A_p|=p^{(n+2)p^n}$. Each quotient ring $R$ of $A_p$ has cardinality at most $p^{(n+2)p^n}$, so
$$\max_{R\in T(p)}|\A^n(R)|=p^{n(n+2)p^n}.$$
Therefore, $B(n)\leq (2^{n(n+2)2^n}!)(3^{n(n+2)3^n}!)$ as desired. The algorithm for deciding periodicity, described in the proof of Theorem \ref{mainthm3}, involves only finitely many instances of multiplication and addition, with the total number of operations needed being bounded by a polynomial function in the number of operations appearing in $S$. Hence, the algorithm is of polynomial time.
\end{proof}

\end{document}